\documentclass[a4paper,10pt,onecolumn]{article}

\usepackage{dblfloatfix}    
\usepackage[textwidth=16cm]{geometry}
\usepackage{graphicx}
\usepackage{caption}
\usepackage{amssymb,amsmath,color,amsthm}
\usepackage{bm}
\usepackage{subcaption}
\usepackage{algorithm}
\usepackage[noend]{algpseudocode}
\usepackage{textcomp}
\usepackage{url}
\usepackage{authblk}
\usepackage{hyperref}

\algnewcommand\algorithmicinput{\textbf{Initialization:}}
\algnewcommand\init{\item[\algorithmicinput]}
\algnewcommand\algorithmicawake{\textsf{\textit{{AWAKE}}}}
\algnewcommand\awake{\item[\algorithmicawake]}
\algnewcommand\algorithmicidle{\textsf{\textit{{IDLE}}}}
\algnewcommand\idle{\item[\algorithmicidle]}
\newcommand{\broad}{{\small \textbf{BROADCAST }}}

\newcommand{\stopp}{{\small \textbf{STOP }}}

\newcommand{\iidle}{\textsf{\textit{{IDLE }}}}
\newcommand{\aawake}{\textsf{\textit{{AWAKE }}}}

\renewcommand{\natural}{{\mathbb{N}}}

\newcommand{\norm}[1]{\|#1\|}
\newcommand{\until}[1]{\{1,\ldots,#1\}}

\newcommand{\EE}{\mathcal{E}}
\newcommand{\GG}{\mathcal{G}}
 
\newcommand{\NN}{\mathcal{N}}

\newcommand{\VV}{\mathcal{V}}

\newcommand{\st}{\text{subject to }}
\newcommand{\m}{\mathop{\rm minimize}}
\newcommand{\R}{\mathbb{R}}
\newcommand{\p}{\bm{p}}
\newcommand{\x}{\bm{x}}
\newcommand{\h}{{h}}
\newcommand{\g}{{g}}
\newcommand{\y}{\bm{y}}
\newcommand{\LL}{\bm{\theta}}

\newcommand{\rhoij}{\rho_{ij}}

\newcommand{\peqi}{\varrho_{i}}
\newcommand{\pini}{\zeta_{i}}
\newcommand{\tik}{{\tau_i^k}}
\newcommand{\tLag}{\tilde{\mathcal{L}}}
\newcommand{\Lag}{\mathcal{L}}
\newcommand{\im}{{i,m}}
\newcommand{\mE}{{m_i^E}}
\newcommand{\mI}{{m_i^I}}

\theoremstyle{plain}
\newtheorem{theorem}{Theorem}
\newtheorem{proposition}[theorem]{Proposition}
\newtheorem{corollary}[theorem]{Corollary}
\newtheorem{lemma}[theorem]{Lemma}

\theoremstyle{definition}
\newtheorem{assumption}{Assumption}

\theoremstyle{remark}
\newtheorem{remark}{Remark}

\newcommand\oprocendsymbol{\hbox{$\square$}}
\newcommand\oprocend{\relax\ifmmode\else\unskip\hfill\fi\oprocendsymbol}

\begin{document}

\title{A Distributed Asynchronous Method of Multipliers\\
for Constrained Nonconvex Optimization
\thanks{This result is part of a
  project that has received funding from the European Research Council (ERC)
  under the European Union's Horizon 2020 research and innovation programme
  (grant agreement No 638992 - OPT4SMART).\newline
  \copyright~2019. This manuscript version is made available under the CC-BY-NC-ND 4.0 license \url{http://creativecommons.org/licenses/by-nc-nd/4.0/}}}

\author[1]{Francesco~Farina}
\author[1]{Andrea~Garulli}
\author[1]{Antonio~Giannitrapani}
\author[2]{Giuseppe~Notarstefano}
\affil[1]{Dipartimento di Ingegneria dell'Informazione e Scienze Matematiche, Universit{\`a} di Siena, Siena, Italy.}
\affil[2]{Department of Electrical, Electronic and Information Engineering ``G. Marconi'', Universit{\`a} di Bologna, Bologna, Italy.}

\date{}

\maketitle

\begin{abstract}
This paper presents a fully asynchronous and distributed approach for 
tackling optimization problems in which both the objective function and 
the constraints may be nonconvex. In the considered network setting each 
node is active upon triggering of a local timer and has access only to a 
portion of the objective function and to a subset of the constraints. In 
the proposed technique, based on the method of multipliers, each node 
performs, when it wakes up, either a descent step on a local augmented 
Lagrangian or an ascent step on the local multiplier vector. Nodes 
realize when to switch from the descent step to the ascent one through 
an asynchronous distributed logic-AND, which detects when all the nodes have reached a predefined tolerance in the minimization of the augmented 
Lagrangian. It is shown that the resulting distributed algorithm is 
equivalent to a block coordinate descent for the minimization of the 
global augmented Lagrangian. This allows one to extend the 
properties of the centralized method of multipliers to the considered 
distributed framework. Two application examples are presented to 
validate the proposed approach: a distributed source localization 
problem and the parameter estimation of a neural network.
\end{abstract}


\section{Introduction}
\label{sec:introduction}
Nonconvex optimization problems are commonly encountered when dealing with control, estimation and learning within cyber-physical networks. In these contexts, typically each device knows only a portion of the whole objective function and a subset of the constraints, so that, to avoid the presence of a central coordinator, distributed algorithms are needed. 

Distributed optimization algorithms handling local constraints are basically designed for convex problems, except for some specific problem settings. In~\cite{lee2013distributed}, the authors propose a distributed random projection algorithm, while a proximal based algorithm is presented in~\cite{margellos2017distributed}. A subgradient projection algorithm has been presented in~\cite{nedic2009distributed} and an extension taking into account communication delays is given in~\cite{LIN2016120}. In~\cite{necoara2013random} randomized block-coordinate descent methods are employed, to solve convex optimization problems with linearly coupled constraints over networks. 

Another relevant class of algorithms is that of distributed primal-dual methods (see, e.g.~\cite{chang2014distributed,yuan2016regularized,SHI201855}). Within this framework, an iterative scheme combining dual decomposition and proximal minimization is introduced in~\cite{FALSONE2017149}.
Distributed approaches based on the Alternating Directions Method of Multipliers (ADMM) are presented in~\cite{iutzeler2016explicit,bianchi2014stochastic,bianchi2016coordinate}.

Asynchronous communication protocols are a typical requirement in real world networks (see, e.g.,~\cite{bertsekas1989parallel} and references therein). Several asynchronous version of distributed optimization algorithms have been proposed in literature, a typical example being the class of gossip-based algorithms~\cite{boyd2006randomized,nedic2011asynchronous}. By building on these works, an asynchronous algorithm based on the Method of Multipliers and accounting for communication failures is introduced in~\cite{jakovetic2011cooperative}. In~\cite{wei20131} an asynchronous ADMM is proposed for a separable, constrained optimization problem. An asynchronous proximal dual algorithm has been presented in~\cite{notarnicola2017asynchronous}.

Distributed algorithms for nonconvex optimization have started to appear in the literature only recently. In~\cite{bianchi2013convergence} a stochastic gradient method is proposed to minimize the sum of smooth nonconvex functions subject to a constraint known to all agents. In~\cite{wai2016projection} a decentralized Frank-Wolfe method for finding a stationary point of the sum of differentiable and nonconvex functions is given. In~\cite{di2016next,sun2016distributed} the authors propose distributed algorithms, respectively for balanced and general directed graphs, based on the idea of tracking the whole function gradient and performing successive convex approximation of the nonconvex cost function. Notice that the approaches in~\cite{bianchi2013convergence,wai2016projection,di2016next,sun2016distributed} do not deal with local constraints, but only with global ones known to all the agents. A perturbed push-sum algorithm for the unconstrained minimization of the sum of nonconvex smooth functions is given in~\cite{tatarenko2017non}. A distributed algorithm dealing with local constraints has been presented in~\cite{notarnicola2016randomized} for a structured class of nonconvex optimization problems. 

The contribution of this paper is a \emph{fully} distributed asynchronous optimization algorithm, hereafter referred to as ASYnchronous Method of Multipliers (ASYMM). The proposed algorithm addresses constrained optimization problems over networks, in which both local cost functions and local constraints may be nonconvex. It features two types of local updates at each node: a primal descent and a multiplier update, which are regulated by an asynchronous distributed logic-AND algorithm.
An interesting feature of ASYMM is that a node does not need to wait for all multiplier updates to start a new primal descent, but rather it just needs to receive the neighbors' multipliers.

The main theoretical result consists in showing that ASYMM implements a suitable inexact version of the Method of Multipliers, in which the primal minimization is performed by means of a block coordinate descent algorithm up to a given tolerance. Thanks to this connection, ASYMM inherits the main properties of the corresponding centralized method~\cite{BERTSEKAS1976133,bertsekas2014constrained}. 
A further contribution is to provide a bound on the norm of the augmented Lagrangian gradient based on the local tolerances, which is instrumental to recover convergence results in the case of inexact primal minimization (see, e.g.,~\cite[Section~2.2.5]{bertsekas2014constrained}).
Finally, it is shown that the proposed algorithm can effectively solve big-data problem (i.e., with a high dimensional decision variable). Indeed, thanks to its block-wise structure, each agent can optimize over and transmit only one block of the entire solution estimate.

The paper is organized as follows.
In Section~\ref{sec:setup} the distributed optimization set-up is presented. The
proposed algorithm is presented in Section~\ref{sec:algorithm} and
analyzed and discussed in Section~\ref{sec:analysis}. In Section~\ref{sec:extensions} an extension for dealing with high dimensional optimization problems is presented. Finally, two numerical applications are
presented in Section~\ref{sec:numerical} and some conclusions are drawn in Section~\ref{sec:conclusions}.

\section{Set-up and Preliminaries}
\label{sec:setup}

\subsection{Notation and definitions}
Given a matrix $A\in\R^{n\times m}$ we denote by
$A[i,j]$ the $(i,j)$-th element of $A$, by $A[:,i]$ its $i$-th column, by $A[i,:]$
its $i$-th row and by $A[i{:}j,k]$ the elements of the $k$-th column of $A$ from row $i$ to $j$. We write $A\left[i,:\right]=b$ to assign the value $b$ to all
the elements in the $i$-th row of $A$. Given two vectors $a,b\in\R^{n}$ and a constant $c$, we write $a>c$ if all elements of $a$
are greater than $c$ and $a>b$ if $a[i]>b[i]$ for all $i$. 
If $J=\{j_1,...,j_m\}$ is a set of indexes, we denote by
$[z_j]_{j\in J}$ the vector $[z_{j_1}^\top,...,z_{j_m}^\top]^\top$.

The following definitions will be useful in the following.  

A function $\Psi(x)$
has \emph{Lipschitz continuous gradient} if there exists a constant $L$ such
that $\|\nabla\Psi(x) - \nabla\Psi(y)\|\leq L\|x-y\|$ for all $x,y$.  It is
\emph{$\sigma$-strongly convex} if
$(\nabla\Psi(x) - \nabla\Psi(y))^\top(x-y)\geq \sigma\|x-y\|^2$.

Let $\x=[x_1^\top,...,x_N^\top]^\top$, with $x_i\in\R^{n}$ and let $U=[U_1,\dots,U_N]$, with $U_i\in\R^{Nn\times n}$ for all $i$, be a partition of the
identity matrix such that $\x = \sum_{i=1}^N U_i x_i$ and $x_i=U_i^\top\x$.
The function $\Phi(\x)$ has \emph{block component-wise Lipschitz continuous gradient} if
there are constants $L_i\geq 0$ such that
$ \|\nabla_{x_i}\Phi(\x+U_i s_i) - \nabla_{x_i}\Phi(\x)\|\leq L_i\|s_i\| $
for all $\x\in \R^{Nn}$ and $s_i\in\R^{n}$.

We say that indexes in
  $\until{N}$ are drawn according to an \emph{essentially cyclic} rule if there
  exists $M\geq N$ such that every $i\in\until{N}$ is drawn at least once every
  $M$ extractions.

\subsection{Distributed Optimization Problem}
Consider the following optimization problem
\begin{equation}\label{pb:problem}
\begin{aligned}
& \m_x
& & \sum_{i=1}^N f_i(x)\\
& \st
& & \h_i(x)=0, & i=1,...,N,\\
& & & \g_i(x)\leq 0, & i=1,...,N,
\end{aligned}
\end{equation}
where $f_i:\R^n\to\R$, $\h_i:\R^n\to\R^\mE$ and $\g_i:\R^n\to\R^\mI$.
Throughout the paper the following assumption is made.
\begin{assumption}\label{asm:C2_functions}
  Functions $f_i$ and each component of $\h_i,\g_i$ are of class $C^2$ and have bounded Hessian. Problem~\eqref{pb:problem} has at least one feasible solution, every local minimum of~\eqref{pb:problem} is a regular point\footnote{A feasible vector $x$ is said to be regular if the gradients of the equality constraints and those of the inequality constraints active at $x$ are linearly independent.} and it satisfies the second order sufficiency conditions. \oprocend
\end{assumption}

The aim of the paper is to present a method for solving problem~\eqref{pb:problem} in a distributed way, by employing a network of $N$ peer
processors without a central coordinator. Each processor has a local memory,
local computation capability and can exchange information with neighboring
nodes. Moreover, functions $f_i$, $\h_i$ and $\g_i$ are private to node $i$.
The network is described by a fixed, undirected and connected graph
$\GG=(\VV,\EE)$, where $\VV=\{1,...,N\}$ is the set of nodes and
$\EE\subseteq\{1,...,N\}\times\{1,...,N\}$ is the set of edges.
We denote by
$\NN_i=\{j\in \VV\mid (i,j)\in\EE\}$ the set of the neighbors of node
$i$ and by $d_i= |\NN_i|+1$. Also, we denote by $d_G$ the diameter of $\GG$.

Regarding the communication protocol, a generalized version of the
asynchronous model presented in~\cite{notarnicola2017asynchronous} is considered. Each node
has its own concept of time defined by a local timer, which triggers when the node has to
awake, independently of the other nodes.  Between two triggering events
each node is in \iidle mode, i.e., it listens for messages from neighboring
nodes and, if needed, updates some local variables, but it does not broadcast
any information. When a trigger occurs, the node switches to \aawake mode, performs
local computations and sends the updated information to neighbors.

\begin{assumption}[Local timers]\label{asm:timing}
  For each node $i$, there exists a constant $\bar{T}_i$ such that node $i$ wakes up at least once in every time interval of length $\bar{T}_i$.\oprocend
\end{assumption}
\begin{assumption}[No simultaneous awakening]\label{asm:no_awake}
  Only one node can be awake at each time instant. \oprocend
\end{assumption}

  Assumption~\ref{asm:timing} implies that in a time interval $\bar{T}=\max_{i\in\until{N}}\bar{T}_i$ each node is awake at least once. Hence, under Assumption~\ref{asm:no_awake}, nodes wake up in an essentially cyclic way. In practice, Assumption~\ref{asm:no_awake} can be relaxed, allowing for non neighboring nodes to be awake in the same time instant, at the price of a slightly more involved analysis of the proposed algorithm.

\subsection{Equivalent Formulation and Method of Multipliers}
In order to solve Problem~\eqref{pb:problem} by means of the above defined network, let us rewrite it in the following form
\begin{equation}\label{pb:distributed_problem}
\begin{aligned}
&\m_{x_1,...,x_N}
& &  \sum_{i=1}^N f_i(x_i)\\
& \st
& & x_i=x_j,& \forall(i,j)\in \EE,\\
& & & \h_i(x_i)=0, & i\in\VV,\\
& & & \g_i(x_i)\leq 0, & i\in\VV.
\end{aligned}
\end{equation}
where $x_i\in\R^n$ for all $i\in\VV$. Notice that, thanks to the connectedness of $\GG$, problems~\eqref{pb:distributed_problem} and~\eqref{pb:problem} are equivalent.

Let us now introduce the augmented Lagrangian associated to
problem~\eqref{pb:distributed_problem}.
Let $\nu_{ij}\in\R^n$ and $\rhoij\in\R$ be the multiplier and penalty parameter
associated to the equality constraint $x_i = x_j$. We compactly define
$\nu_i=[\nu_{ij}]_{j\in\NN_i}$, $\rho_i=[\rhoij]_{j\in\NN_i}$. Similarly, let $\lambda_{i}\in\R^\mE$
and $\peqi\in\R$ ( respectively $\mu_{i}\in\R^\mI$ and $\pini\in\R$) be the multiplier and penalty
parameter associated to the equality (respectively inequality)
constraint of node $i$.
Moreover, let $\x=[x_1^\top,...,x_N^\top]^\top$; denote by $\p=[\rho_i,\peqi,\pini]_{i\in\VV}$ the vector
stacking all the penalty parameters; $\nu=[\nu_i]_{i\in\VV}$,
$\lambda=[\lambda_i]_{i\in\VV}$ and $\mu=[\mu_i]_{i\in\VV}$
be the vectors stacking the corresponding multipliers, and, consistently, let
$\LL=[\nu^\top,\lambda^\top,\mu^\top]^\top$. 
Let us define for notational convenience
\begin{equation}\label{eq:q_operator}
	q_{c}(a,b)=\frac{1}{2c}\left(\max\{0, a+cb\}^2 -a^2\right).
\end{equation}
When $a$ and $b$ are vectors, the right hand side in~\eqref{eq:q_operator} is intended component-wise.
Then, the augmented Lagrangian
associated to~\eqref{pb:distributed_problem} is defined as
\begin{align}
\Lag_{\p}(\x,\LL)=&\sum_{i=1}^N \bigg\lbrace f_i(x_i) \nonumber\\
&+\sum_{j\in \NN_i}\left[\nu_{ij}^\top(x_i-x_j)+\frac{\rhoij}{2}\|x_i-x_j\|^2\right]+\nonumber\\
&+\lambda_{i}^\top \h_{i}(x_i)+\frac{\peqi}{2}\| \h_{i}(x_i)\|^2+\nonumber\\
&+\bm{1}^\top q_{\pini}(\mu_i, \g_i(x_i))
\bigg\rbrace.\label{eq:L}
\end{align}

Notice that, more generally, one can associate a different penalty parameter to each component of the equality and inequality constraints of each node. This extension is omitted in order to streamline the presentation.

A powerful method for solving problem~\eqref{pb:distributed_problem} is the well
known Method of Multipliers, which consists of the following steps (see
e.g. \cite{rockafellar1974augmented, bertsekas2014constrained}),
\begin{align}
	\x^{k+1} &= \arg \min_{\x} \Lag_{\p^k}(\x,\LL^k)\label{eq:x_minimization}\\
	\nu_{ij}^{k+1} &= \nu_{ij}^k+\rho_{ij}^k(x_i^{k+1}-x_j^{k+1}), &\forall(i,j)\in\EE,\label{eq:nu}\\
	\lambda_{i}^{k+1} &= \lambda_{i}^k + \peqi^k \h_{i}(x_i^{k+1}),&\forall i\in\VV,\\
	\mu_{i}^{k+1} &= \max\{0,\, \mu_{i}^k+\pini^k \g_{i}(x_i^{k+1})\},& \forall i\in\VV,\label{eq:mu}
\end{align}
where the max operator is to be intended component-wise and $\p^{k+1}\geq\p^k\geq ...\geq \p^0> 0$. 

A typical update rule for a penalty parameter $\rhoij$ associated to an equality constraint $x_i=x_j$ is 
\begin{equation}\label{eq:penalty_eq_x}
  \rhoij^{k+1}=
  \begin{cases}
    \beta \rhoij^k,&\text{if }\norm{x_i^{k+1}-x_j^{k+1}}>\gamma \norm{x_i^k-x_j^k},\\
    \rhoij^k,&\text{otherwise},
  \end{cases}
\end{equation}
where $\beta$ and $\gamma$ are positive constants (see~\cite[Section~2.2.5]{bertsekas2014constrained}). Similarly, the update rule
for a penalty parameter $\peqi$ associated to an equality constraint $h_i(x_i)=0$ is
\begin{equation}\label{eq:penalty_eq}
  \peqi^{k+1}=
  \begin{cases}
    \beta \peqi^k,&\text{if }\norm{h_i(x_i^{k+1})}>\gamma \norm{h_i(x_i^k)},\\
    \peqi^k,&\text{otherwise},
  \end{cases}
\end{equation}
while the rule for a penalty parameter $\pini$ associated to an inequality constraint $g_i(x_i)\leq 0$ is
\begin{equation}\label{eq:penalty_ineq}
  \pini^{k+1}=
  \begin{cases}
    \beta\pini^k,&\text{if }\norm{g_i^+(x_i^{k+1},\mu_i^{k+1},\zeta_i^k)}>\\
    & \qquad>\gamma\norm{g_i^+(x_i^k,\mu_i^{k},\zeta_i^k)},\\
    \pini^k,&\text{otherwise},
  \end{cases}
\end{equation}
where $g_i^+(x_i,\mu_i,\zeta_i)=\max\{g_i(x_i),-\frac{\mu_i}{\zeta_i}\}$.

The minimization step~\eqref{eq:x_minimization} can be carried out approximately at each step $k$, up to a certain precision $\varepsilon^k$. If the sequence $\{\varepsilon^k\}\to 0$ as $k\to\infty$, the minimization step is said \emph{asymptotically exact} (see~\cite[Section~2.5]{bertsekas2014constrained}).

Sufficient conditions guaranteeing the convergence of method~\eqref{eq:x_minimization}-\eqref{eq:mu} to a local
minimum of problem~\eqref{pb:distributed_problem} have been given, e.g., in
\cite{bertsekas2014constrained}.  One of these conditions involves the
regularity of the local minima of the optimization problem. In general, such a
condition is not verified in problem~\eqref{pb:distributed_problem} due to the
constraints $x_i=x_j$ for all
$(i,j)\in\EE$. In~\cite{matei2014extension,matei2015distributed} the results
in~\cite{bertsekas2014constrained} have been extended to deal with the non
regularity of the local minima of problem~\eqref{pb:distributed_problem}. With respect to those works, the main novelty of the solution proposed in this paper is that the network model is asynchronous and the switching between a primal and a multiplier update is performed by the nodes in a fully distributed way.

\section{Asynchronous Method of Multipliers}
\label{sec:algorithm}
In this section, the Asynchronous Method of Multipliers (ASYMM) for solving problem~\eqref{pb:distributed_problem} in an asynchronous
and distributed way is presented. Let us first present a distributed
algorithm whose aim is to check whether all nodes in an asynchronous network have set a local flag to one. It can be seen as the asynchronous counterpart of the synchronous logic-AND algorithm presented in~\cite{ayken2015diffusion}.

\subsection{Asynchronous distributed logic-AND}\label{subsec:AND}
Each node in the network is assigned a flag
$C_i$ that is initially set to $0$ and is then changed to $1$ in finite time.
The aim of the asynchronous distributed logic-AND algorithm is to check if all
the nodes have $C_i=1$.

Each node $i$ stores a matrix $S_i\in\{0,1\}^{d_G\times d_i}$ which contains
information about the status of the node itself and its neighbors.  Let
$S_i[l,j_i]$ denote the element in the $l$-th row and $j_i$-th column of
$S_i$, where $j_i$ is the index associated to node $j$ by node $i$. The
elements $S_i[1,j_i]$ for $j\in \NN_i$ represent the values of the
flags of nodes $j\in\NN_i$ and $S_i[1,d_i]$ the one of node $i$
itself. This means that $S_i[1,d_i]=C_i$.
Moreover, for $l=2,\dots,d_G$, the element $S_i[l,j_i]$,
$j\in \mathcal{N}_i$, contains the status of the $(l-1)$-th row of
$S_j$, which is defined as the product of all its entries. Similarly, $S_i[l,d_i]$
contains the status of the $(l-1)$-th row of $S_i$ and it is computed as
\begin{equation}
S_i[l,d_i]=\prod_{b=1}^{d_i}S_i[l-1,b].
\end{equation}
Hence, one has that $S_i[l,d_i]=1$ if and only if $S_i[l-1,j_i]=1$ for
all $j\in \NN_i$ and $S_i[l-1,d_i]=1$.
\begin{algorithm}
	\begin{algorithmic}
          \init $C_i\gets 0$, 
          $S_i \gets \bm{0}_{d_G\times d_i}$
		\item[]
		\awake
		\If{$\prod_{b=1}^{d_i}S_i[d_G,b]\neq1$}
		\State $S_i[1,d_i]\gets C_i$
		\State $S_i[l,d_i]\gets\prod_{b=1}^{d_i}S_i[l-1,b]$ for $l=2,...,d_G$
		\State \broad $S_i[:,d_i]$ to all $j\in \NN_i$
		\EndIf
		\item[]
		\If{$\prod_{b=1}^{d_i}S_i[d_G,b]=1$}
		\State \stopp and send \stopp signal to all $j \in \NN_i$
		\EndIf
              \item[] \idle \If{$S_j[:,d_j]$ received from
                  $j\in \NN_i$ and not received a \stopp
                  signal} 
		\State $S_i[l,j_i]\gets S_j[l,d_j]$ for $l=1,...,d_G$
		\EndIf
		\State \textbf{if} \stopp received, set $S_i[d_G,:]\gets 1$

	\end{algorithmic}
	\caption{Asynchronous distributed logic-AND}\label{alg:async_AND}
\end{algorithm}

A pseudo code of the distributed logic-AND algorithm is reported in Algorithm~\ref{alg:async_AND}.
Notice, in particular, that node $i$ has to broadcast to all
its neighbors only the last column of $S_i$, i.e. $S_i[l,d_i]$ for
$l=1,...,d_G$. Moreover, it stores only the $d_j$-th column of the matrices $S_j$
of its neighbors $j\in \NN_i$, whenever it receives them.

It is apparent that a node will stop only when the last row of its
matrix $S_i$ is composed by all $1$s, i.e. when
\begin{equation}\label{eq:stop_S_equation}
\prod_{b=1}^{d_i}S_i[d_G,b]=1.
\end{equation}

In the following result it is shown that~\eqref{eq:stop_S_equation} is satisfied
at some node if and only if $C_i=1$ for all $i$.

\begin{proposition}
\label{prop:logic-AND}
  Let Assumption~\ref{asm:timing} hold.
If there exists a time instant after which $C_j{=}1$ indefinitely for all $j\in\VV$, then $\prod_{b=1}^{d_\ell}S_\ell[d_G,b]=1$ in finite time for all nodes $\ell\in\VV$.
Conversely, if  there exists a node $\ell$ satisfying $\prod_{b=1}^{d_\ell}S_\ell[d_G,b]=1$ at a certain time instant, then every node $j\in\VV$ must have had $C_j=1$ at some previous time instant.
\end{proposition}
\begin{proof}
In a time interval of $\bar{T}$, every node wakes up at least
  once. In the worst case, in which the distance between two generic nodes $j$ and
  $\ell$ is equal to the graph diameter ($d_G$), in a time
  $d_G\bar{T}$ there exists an ordered subsequence of
  awakenings following the path $j\to \ell$. Hence, if $C_i=1$ $\forall i$, node $h$
  along this path will broadcast $S_h[:,d_h]=1$ to its neighbors and $S_\ell[d_G,:]$ will eventually contain only $1$s, thus leading to 
  $\prod_{b=1}^{d_\ell}S_\ell[d_G,b]=1$.
  Now assume that $\prod_{b=1}^{d_\ell}S_\ell[d_G,b]=1$ at some time instant and suppose, by contradiction, that there exists some node $j$ for which $C_j=0$ at all previous time instants. By assumption, all nodes
  $i\in \NN_j$ have $S_i[1, j_i]=0$ (because columns sent
  by $j$ always contained a zero in that position), which in turn, implies that
  $S_i[2{:}d_G, d_i]=0$ for all $i\in \NN_j$. Then, for every
  $i\in \NN_j$, every $m\in \NN_i$ have
  $S_m[2, i_m]=0$ and hence $S_m[3:d_G, d_m]=0$.  By induction, node $\ell$ must
  have at least one element of its $d_G$-th row equal to $0$, which contradicts
  the assumption and hence completes the proof.
\end{proof}

Notice that the only information that the nodes need to know about the network is the graph diameter $d_G$.
It is worth recalling that such a parameter can be preliminary computed in a distributed way (see, e.g.,~\cite{oliva2016distributed} and references therein). Furthermore, only an upper bound on $d_G$ is necessary to run Algorithm~\ref{alg:async_AND}, at the price of an increase in the time needed in order to achieve the termination condition~\eqref{eq:stop_S_equation}.

\subsection{Asynchronous distributed optimization algorithm}
It is worth stressing that the augmented Lagrangian defined in~\eqref{eq:L} is not
separable in the local decision variables $x_i$. Thus, the minimization step
in~\eqref{eq:x_minimization} cannot be performed by independently minimizing
with respect to each variable. 
In order to devise a distributed algorithm for solving
problem~\eqref{pb:distributed_problem} it is useful to define a \emph{local augmented
Lagrangian}, whose minimization with respect to the decision variable $x_i$ is
equivalent to the minimization of the entire augmented Lagrangian
\eqref{eq:L}. To this aim, let $x_{\NN_i} = [x_{j}]_{j\in \NN_i\cup\{i\}}$,
$\LL_{\NN_i}=[\lambda_i,\mu_i,\nu_{i},[\nu_{ji}]_{j\in \NN_i}]$,
and $\p_{\NN_i}=[\peqi,\pini,\rho_i, [\rho_{ji}]_{j\in\NN_i}]$.
Then, the $i$-th local augmented Lagrangian, which groups
together all the terms of~\eqref{eq:L} depending on $x_i$, is defined as
\begin{align}
\tLag_{\p_{\NN_i}}(x_{\NN_i},\LL_{\NN_i})
&= f_i(x_i)+ \nonumber\\
&+\sum_{j\in \NN_i}\left[x_i^\top(\nu_{ij}-\nu_{ji})+\frac{\rhoij+\rho_{ji}}{2}\|x_i-x_j\|^2\right]+\nonumber\\
&+\lambda_{i}^\top \h_{i}(x_i)+\frac{\peqi}{2}\|\h_{i}(x_i)\|^2+\nonumber\\
&+\bm{1}^\top q_{\pini}(\mu_i, \g_i(x_i)).\label{eq:L_local}
\end{align}
The following proposition holds.
\begin{proposition}\label{prop:Lipschitz} 
Let Assumption~\ref{asm:C2_functions} hold. Then 
\begin{equation}\label{eq:gradient_equivalence}
	\nabla_{x_i}\Lag_{\p}(\x,\LL) = \nabla_{x_i}\tLag_{\p_{\NN_i}}(x_{\NN_i},\LL_{\NN_i}).
\end{equation}
Also, for fixed values of $x_j$, $j\neq i$,
\begin{equation}\label{eq:min_equivalence}
	\arg \min_{x_i}\Lag_{\p}(\x,\LL)= \arg\min_{x_i}\tLag_{\p_{\NN_i}}(x_{\NN_i},\LL_{\NN_i}).
\end{equation}
Moreover, $\tLag_{\p_{\NN_i}}(x_{\NN_i},\LL_{\NN_i})$ has Lipschitz continuous gradient for all $i\in\VV$.
\end{proposition}
\begin{proof}
From~\eqref{eq:L} and~\eqref{eq:L_local} it can be easily seen that 
\begin{equation}
  \Lag_{\p}(\x,\LL)=\tLag_{\p_{\NN_i}}(x_{\NN_i},\LL_{\NN_i})+\Psi(x_{-i},\LL_{-i})
\end{equation}
where $\Psi(x_{-i},\LL_{-i})$ is a function which does \emph{not} depend on local variables of node $i$. Hence~\eqref{eq:gradient_equivalence} and~\eqref{eq:min_equivalence} follow.
By Assumption~\ref{asm:C2_functions}
$\Lag_{\p}(\x,\LL)\in C^2$ in the set
$\{\x\mid g_i(x_i)\neq -\mu_i/\zeta_i, \forall i\in\VV\}$ for all $\LL$ and
$\p>0$ (see e.g.~\cite{bertsekas2014constrained}, Proposition 3.1). Hence, one
has that $\nabla_{\x}\Lag_{\p}(\x,\LL)$ is almost everywhere differentiable for
all $\LL$ and $\p>0$. Moreover, from Assumption~\ref{asm:C2_functions}, it holds
that $\nabla_{\x\x}\Lag_{\p}(\x,\LL)$ is bounded, and $\Lag_{\p}(\x,\LL)$ has
Lipschitz continuous gradient for all $\LL$ and $\p>0$. Hence $\Lag_{\p}(\x,\LL)$ has block component-wise Lipschitz continuous gradients and, from~\eqref{eq:gradient_equivalence} $\tLag_{\p_{\NN_i}}(x_{\NN_i},\LL_{\NN_i})$ has Lipschitz continuous gradient.
\end{proof}

The ASYMM algorithm for solving problem~\eqref{pb:distributed_problem} in an asynchronous and distributed way is now introduced. When a node wakes up, it performs either one gradient descent step on its
local augmented Lagrangian or a multiplier update.
The nodes keep performing gradient descent steps, until all of them have reached a prescribed tolerance on the norm of the local augmented Lagrangian gradient.
This check is performed by nodes themselves in a distributed way, by using the logic-AND algorithm presented in Section~\ref{subsec:AND}. When a node gets aware of this condition, it performs one ascent step on its local
multiplier vector. After it has received
the updated multipliers from all its neighbors, it gets back to the primal update.

More formally, when node $i$ wakes up, it checks through a flag, called $M_{done}$,
if its multiplier vector and the neighboring ones are up to date. If this is the case (which corresponds to $M_{done}=0$),
it performs one of the following two tasks:
\begin{figure*}[!b]
	\centering
	\includegraphics[width=\linewidth]{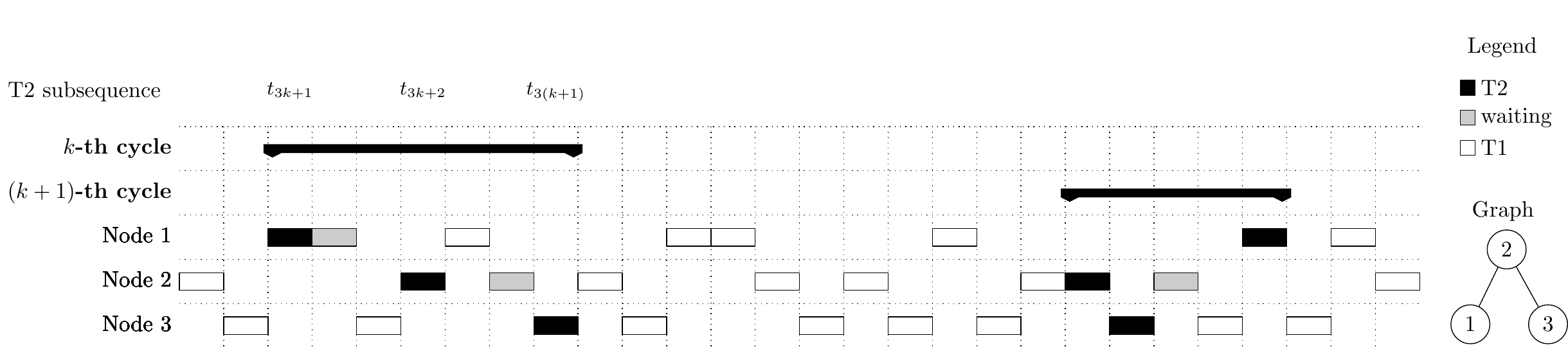}
	\caption{An example of the execution of ASYMM for a network with three nodes.}
	\label{fig:gantt}
\end{figure*}
\begin{enumerate}
\item[T1.] \label{task:1} If $\prod_{l=1}^{d_i}S_i[d_G,l]\neq 1$, node $i$ performs a
  gradient descent step on its local augmented Lagrangian (using $1/L_i$ as stepsize, where $L_i$ is the Lipschitz constant of $\tLag_{\p_{\NN_i}}(x_{\NN_i},\LL_{\NN_i})$)
  and checks if the
  local tolerance $\epsilon_i>0$ on the gradient has been reached. If the latter is true, it corresponds to setting
  $C_i\gets 1$ in the distributed logic-AND Algorithm~\ref{alg:async_AND}. Then,
  it updates matrix $S_i$, and broadcasts the updated $x_i$ and the column $S_i[:,d_i]$ to its
  neighbors.
\item[T2.] \label{task:2} If $\prod_{l=1}^{d_i}S_i[d_G,l] = 1$, node $i$ performs an
  ascent step on the local multiplier vector and updates the local penalty
  parameters according to equations~\eqref{eq:nu}-\eqref{eq:mu} and~\eqref{eq:penalty_eq_x}-\eqref{eq:penalty_ineq}, respectively. Then, it sets $M_{done}=1$ and broadcasts the updated multipliers and
  penalty parameters $\nu_{ij}$ and $\rho_{ij}$ (associated to constraints
  $x_i=x_j$) to its neighbors.
\end{enumerate}

\begin{algorithm}
	\caption{ASYMM}\label{alg:async}
	\begin{algorithmic}
		\init Initialize $x_i$, $\LL_i$, $\NN_i$, $\p_i$, $S_i=\bm{0}_{d_G\times d_i}$, $M_{done}$ = 0.
		\item[]

		\awake
		\If{$\prod_{b=1}^{d_i}S_i[d_G,b]\neq 1$ \textbf{and} \textbf{not} $M_{done}$}\\
		\vspace{-1ex}
		\State $x_i\gets x_i-\frac{1}{L_i}\nabla_{x_i}\tLag_{\p_{\NN_i}}(x_{\NN_i},\LL_{\NN_i})$\\
		\vspace{-1ex}
		\If{$\|\nabla_{x_i}\tLag_{\p_{\NN_i}}(x_{\NN_i},\LL_{\NN_i})\|\leq\epsilon_i$} $S_i[1,d_i]\gets 1$\\
		\vspace{-1.5ex}
		\EndIf
		\State $S_i[l,d_i]\gets\prod_{b=1}^{d_i}S_i[l-1,b]$ for $l=2,...,d_G$\\
		\vspace{-1.5ex}
		\State \broad $x_i$, $S_i[:,d_i]$ to all $j\in \NN_i$
		\EndIf
		\item[]
		\If{$\prod_{b=1}^{d_i}S_i[d_G,b]=1$ \textbf{and} \textbf{not} $M_{done}$}\\
		\vspace{-1ex}
		\State $\nu_{ij}\gets\nu_{ij}+\rho_{ij}(x_i-x_j)$ for $j\in \NN_i$\\
		\vspace{-2ex}
		\State $\lambda_{i}\gets \lambda_{i}+\peqi \h_{i}(x_i)$\\
		\vspace{-2ex}
		\State $\mu_{i}\gets \max\{0,\, \mu_{i}+\pini \g_{i}(x_i)\}$\\
		\vspace{-1ex}
		\State update $\peqi$, $\pini$ and $\rho_i$
		\State $M_{done}$ $\gets$ 1
		\State \broad $\nu_{ij}$, $\rhoij$ to $j \in \NN_i$
		\EndIf
		\item[]

		\idle
		\If{$S_j[:,d_j]$ received from $j\in \NN_i$ and not already received some new $\nu_{ji}$}
		\State $S_i[l,j_i]\gets S_j[l,d_j]$ for $l=1,...,d_G$
		\EndIf
		\State \textbf{if} $\nu_{ji}$ and $\rho_{ji}$ received from $j\in \NN_i$  set $S_i\left[d_G,:\right]\gets 1$
		\State \textbf{if} $x_j^{new}$ received from $j\in \NN_i$, update $x_j\gets x_j^{new}$
		\If{$M_{done}$ \textbf{and} $\nu_{ji}$ received from all $j\in \NN_i$}
		\State $M_{done}$ $\gets$ 0, $S_i\gets\bm{0}_{d_G\times d_i}$, update $\epsilon_i$
		\EndIf
	\end{algorithmic}
\end{algorithm}

When in \iidle, node $i$ continuously listens for messages from its neighbors,
but does not broadcast any information. Received messages may contain either
local optimization and logic-AND variables, or multiplier
vectors and penalty parameters. If necessary, node $i$ suitably updates local logic-AND variables or the $M_{done}$ flag. Notice that, for node $i$,
sending a new multiplier $\nu_{ij}$ or receiving a new $\nu_{ji}$ corresponds to
sending or receiving a \stopp signal in the asynchronous logic-AND algorithm. 

The ASYMM pseudocode is reported in Algorithm~\ref{alg:async} and an example of its execution is shown in Figure~\ref{fig:gantt}, where tasks T1 and T2 are denoted by white and black blocks, respectively.

\begin{remark}
  A gossip-based distributed algorithm based on the Method of Multipliers for convex optimization problems has been proposed in~\cite{jakovetic2011cooperative}. In ASYMM the logic-AND allows the nodes to perform multipliers updates asynchronously, while in~\cite{jakovetic2011cooperative} a global clock is employed to regulate such an update in a synchronous way.
  One main advantage of ASYMM is that it guarantees that a prescribed tolerance is reached by each node in the primal descent, which in turn is a crucial feature when solving nonconvex optimization problems.
\end{remark}

\section{ASYMM Analysis}
\label{sec:analysis}
In order to analyze the ASYMM algorithm, we start by noting that under
Assumptions~\ref{asm:timing} and~\ref{asm:no_awake}, from a global perspective,
the local asynchronous updates can be treated as an algorithmic evolution in which,
at each iteration, only one node wakes up in an essentially cyclic fashion.

Given the above, it is possible to associate an iteration of
the distributed algorithm to each triggering.
Denote by $t\in\natural$ a discrete,
\emph{universal} time indicating the $t$-th iteration of the algorithm and
define as $i_t\in\VV$ the index of the node triggered at iteration
$t$.

In the following, it will be shown that: (i) there is an \emph{equivalence} relationship
between ASYMM and an inexact Method of Multipliers and (ii) under suitable technical conditions, a bound on the gradient of the augmented Lagrangian can be derived from the local tolerances $\epsilon_i$.

\subsection{Equivalence with an inexact Method of Multipliers}
Consider an inexact Method of Multipliers which consists of solving the $k$-th
instance of the augmented Lagrangian minimization by means of a block-coordinate
gradient descent algorithm (see, e.g., \cite{wright2015coordinate} for a
survey), which runs for a certain number of iterations $h^k$. A pseudo code of
this inexact Method of Multipliers (inexact MM) is given in Algorithm~\ref{alg:inexactMM}, where
$i_{h}$ is the index of the block chosen at iteration $h$ and the penalty parameters are updated as in~\eqref{eq:penalty_eq}-\eqref{eq:penalty_ineq}.
\begin{algorithm}
	\begin{algorithmic}
		\For{$k=0,1,...$}
		\State $\hat{\x}^0=\x^k$
		\For{$h =1,...,h^k$}
		\State $\hat{\x}^{h+1}=\hat{\x}^h-\frac{1}{L_{i_h}^k} U_{i_h}\nabla_{x_{i_h}}\Lag_{\p}(\hat{\x}^h,\LL^k)$
		\EndFor
		\State $\x^{k+1}=\hat{\x}^{h^k+1}$
		\State $\nu_{ij}^{k+1} = \nu_{ij}^k+\rho_{ij}^k(x_i^{k+1}-x_j^{k+1}),\quad \forall(i,j)\in\EE$
		\State $\lambda_{i}^{k+1} = \lambda_{i}^k + \peqi^k \h_{i}(x_i^{k+1}),\forall i\in\VV$
		\State $\mu_{i}^{k+1} = \max\{0,\, \mu_{i}^k+\pini^k \g_{i}(x_i^{k+1})\},\quad \forall i\in\VV$
		\EndFor
	\end{algorithmic}
	\caption{Inexact MM}\label{alg:inexactMM}
\end{algorithm}

It is worth remarking that the ordered sequence of indexes $h$ and $k$ used
in Algorithm~\ref{alg:inexactMM} does not coincide with the sequence of
universal times $t$ of ASYMM. It will be rather shown that a (possibly reordered)
subsequence of iterations in the universal time $t$ of ASYMM gives rise
to suitable $h$ and $k$ sequences in Algorithm~\ref{alg:inexactMM}.

Let $t_1,t_2,...$ be a subsequence of $\{t\}$ such that at each $t_\ell$ a
multiplier update (task T2) has been performed by node $i_{t_\ell}$ and let $t_1$
be the time instant of the first multiplier update. Then, the following result
holds.
\begin{lemma}\label{lemma:l_cyc} %
  Each sequence $(i_{t_{kN+1}},...,i_{t_{(k+1)N}})$, for $k=0,1,...$, is a
  permutation of $\{1,...,N\}$. Moreover, if $\epsilon_i>0$ for all $i\in\VV$,
  multiplier updates occur infinitely many times.
\end{lemma}
\begin{proof}
  Consider the first multiplier update, performed by node
  $i_{t_{1}}$ at $t_{1}$. Then, until all other nodes have performed their
  first multiplier update, there will be some node $j$ which has not received back all
  the new multipliers from its neighbors and has $M_{done}=1$. Hence, it cannot
  run task T1, and consequently it cannot set and broadcast $S_j[1,d_j]=1$.  So,
  node $i_{t_{1}}$ has at least one element of the last row of
  $S_{i_{t_{1}}}$ at $0$, hence it cannot perform another multiplier update
  (although it could have started over performing task T1). The first part of the proof is completed by induction.
  In order to prove the second part of the lemma, assume, by contradiction, that
  the number of multiplier updates is finite and denote by $t_M$ the time
  instant of the last multiplier update. If $\text{mod}(M,N) \neq 0$, from the
  connectedness of $\GG$ at $t_M+1$ there exists at least one node $j$ that: i)
  has not updated its multipliers yet; ii) has a neighbor who has already
  performed a multiplier update. Hence, node $j$ will perform a multiplier update
  next time it wakes up (which occurs in finite time by Assumption~\ref{asm:timing}), thus
  contradicting the assumption that $t_M$ was the time instant of the last
  multiplier update. If $\text{mod}(M,N) = 0$ (i.e.,
    $(i_{t_{M-N+1}}, \dots, i_{t_{M-1}}, i_{t_M})$ is a permutation of
    $\until{N}$), all the nodes that wake up after $t_M$ will run task T1. From
    a global perspective, this can be seen as a block coordinate descent
    algorithm on the augmented Lagrangian with a given multiplier vector. Since
    this algorithm converges to a stationary point~\cite{xu2017globally}, every
    node $i \in \VV$ will reach its local tolerance $\epsilon_i > 0$ in finite
    time and then set $S_i[1,d_i] =1$. From Proposition~\ref{prop:logic-AND}, after a finite
  number of iterations some node $j$ will satisfy
  $\prod_{b=1}^{d_j} S_j[d_G,b] = 1$ and hence it will run a new multiplier
  update, which contradicts the assumption that $t_M$ was the time instant of
  the last multiplier update.
\end{proof}

In the sequel the subset of universal times $t$,
$\{{t_{kN+1}},...,{t_{(k+1)N}}\}$, during which $N$ tasks T2 are performed, will be referred to as the $k$-th \emph{cycle} of ASYMM. 

The example in Figure~\ref{fig:gantt} shows the $k$-th and $(k{+}1)$-th cycles of an ASYMM run. According to Lemma~\ref{lemma:l_cyc}, during a single cycle each node performs task T2 once. It is worth remarking that in the $k$-th cycle node $1$ starts over the $(k{+}1)$-th primal minimization before node $3$ has completed its $k$-th multiplier update. This happens because node $1$ has already received the updated multipliers and penalty parameters from node $2$, which is the only neighbor of node $1$. The same thing happens to node $3$ in the $(k{+}1)$-th cycle. This is a key feature of the asynchronous distributed scheme underlying ASYMM. It can also be observed that when node $3$ wakes up for the first time, after node $1$ has done the $k$-th dual update, it performs again task T1. In fact, node $3$ has not received a STOP signal yet, because it is not connected directly to node $1$.

Define $\LL_i=[\lambda_i,\mu_i,\nu_{i}]$ and let
$\tilde{x}_{i}^t$ and $\tilde{\LL}_{i}^t$ be the value of the state vector and of the
multiplier vector of node $i$, at iteration $t$, computed according to ASYMM.  Then, the following
Corollary holds, whose proof follows immediately from Lemma~\ref{lemma:l_cyc}.

\begin{corollary}
\label{cor:const_lambda}
  Let $\tau\in\{t_{kN+1},...,t_{(k+1)N}\}$, for some $k=0,1,...$. Then one has
  $\tilde{\LL}_{i_\tau}^{t}=\tilde{\LL}_{i_\tau}^{\tau}$ for all $t=\tau,\tau{+}1,\dots,t_{(k+1)N}$. \oprocend
\end{corollary}
Corollary~\ref{cor:const_lambda} states that once a multiplier vector is updated in a cycle, then, its value remains unchanged for the whole cycle.

Let $\tik$ be the time instant in which node $i$ performs task T2 in the $k$-th cycle, i.e.,
$\tik \in \{t_{kN+1},...,t_{(k+1)N}\}$ such that node $i$ is awake at time
$\tik$. 
By using this Corollary~\ref{cor:const_lambda}, one can define
 \begin{equation*}
 	x_{i}^{k+1}=\tilde{x}_{i}^{\tau_i^k},\quad
 	\LL_{i}^{k+1}=\tilde{\LL}_{i}^{\tau_i^k},
 \end{equation*}
$k=0,1,...$ and by using Lemma~\ref{lemma:l_cyc} and reordering the indexes $i_{\tau}$,
\begin{align*}
	\x^{k+1}&=\left[\left(x_1^{k+1}\right)^\top,...,\left(x_N^{k+1}\right)^\top\right]^\top,\\
	\LL^{k+1}&=\left[\left(\LL_1^{k+1}\right)^\top,...,\left(\LL_N^{k+1}\right)^\top\right]^\top.
\end{align*}

The next two lemmas show that a local primal (resp. multiplier) update is performed
according to a common multiplier (resp. primal) variable.
\begin{lemma}\label{lemma:l_consistency}
  For all $\tau\in\{t_{kN+1},...,t_{(k+1)N}\}$, $k=0,\!1,...$, every multiplier
  $\LL_{i_\tau}^{k+1}$ is computed using the state vector $\x^{k+1}$.
\end{lemma}
\begin{proof}
  Consider $\tau\in\{t_{kN+1},...,t_{(k+1)N}\}$ for a given $k$. Node $i_\tau$ computes
  \begin{align*}
    x_{i_\tau}^{k+1}&=\tilde{x}_{i_\tau}^\tau,\\
    \nu_{i_\tau j}^{k+1}&=\nu_{i_\tau j}^k+\rho_{i_\tau j}^k(x_{i_\tau}^{k+1}-\tilde{x}_j^{\tau}),&\forall j\in \NN_{i_\tau}\label{eq:nu_update},\\
    \lambda_{i_\tau}^{k+1}&=\lambda_{i_\tau}^k + \varrho_{i_\tau}^k \h_{i_\tau}(x_{i_\tau}^{k+1}),\\
    \mu_{i_\tau}^{k+1}&=\max\{0,\, \mu_{i_\tau}^k+\zeta_{i_\tau}^k \g_{i_\tau}(x_{i_\tau}^{k+1})\}.
  \end{align*}
  First, notice that the update of node $i_\tau$ depends only on $x_{j}$ or
  $\LL_{j}$ with $j\in \NN_{i_\tau}$.
  Then, let us show that $\tilde{x}_j^{\tau} = x_j^{k+1}$ for all $j\in \NN_{i_\tau}$.
  If a node $j\in \NN_{i_\tau}$ has already performed its multiplier update in
  the $k$-th cycle,  
  then, even if it woke up again before time $\tau$ it
  did not update $x_j$ (did not start a new primal update) because it has not
  received all the updated multipliers from its neighbors (node $i_\tau$ has not performed the
  multiplier update yet and thus has not sent $\nu_{i_\tau j}^{k+1}$ to node $j$).
  Therefore, $\tilde{x}_j^{\tau} = x_j^{k+1}$. If, vice-versa, node $j\in \NN_{i_\tau}$
  has not performed its multiplier update in
  the $k$-th cycle, 
  then $\tilde{x}_j^{\tau}$ will become $x_j^{k+1}$ next time node $j$ wakes up (because node
  $i_\tau$ has sent it the updated multiplier while it was in idle, so that $j$
  has set $S_j[d_G,:]=1$).
\end{proof}

\begin{lemma}\label{lemma:x_consistncy}
  For all $t=\tau_i^k,\tau_i^k+1,\dots,\tau_i^{k+1}$, every
  descent step on the augmented Lagrangian with respect to the block-coordinate
  $x_i$ is performed using the multiplier vector $\LL^{k+1}$.
\end{lemma}
\begin{proof}
  Node $i$ can start over a block coordinate descent iteration after $\tau_i^k$
  only when it has received all the new multipliers $\nu_{ji}^{k+1}$ (and the
  corresponding penalty parameters) from its neighbors. Thanks to~\eqref{eq:gradient_equivalence}, it is sufficient to show that each local descent on the $i$-th local augmented Lagrangian is performed using the multiplier vector $\LL_{\NN_i}^{k+1}$. This follows from Lemma~\ref{lemma:l_cyc} and Corollary~\ref{cor:const_lambda} by using arguments similar to those in the proof of Lemma~\eqref{lemma:l_consistency}.
  \end{proof}

Next lemma states that every node performs at least one primal update between
the beginning of two consecutive cycles.
\begin{lemma}\label{lemma:awake_once}
  Between $t_{kN+1}$ and $t_{(k+1)N+1}$ every node performs task $T1$ at least
  once.
\end{lemma}
\begin{proof}
	Since at iteration $t_{(k+1)N}$ all the nodes have performed the $k$-th multiplier update, each of them has set $S_i[1, d_i] = 0$ at some time between $t_{kN+1}$ and $t_{(k+1)N}$. At time $t_{(k+1)N+1}$ the first node performs the $(k + 1)$-th multiplier update. This can occur only if each node $i$ has set $S_i[1,d_i]=1$ at some time between $t_{kN+1}$ and $t_{(k+1)N+1}$, which implies that each node performs task T1 at least once over the same time interval.
\end{proof}

The equivalence of ASYMM and Algorithm~\ref{alg:inexactMM} is stated by the next theorem.
\begin{theorem}\label{thm:equivalence}
  Let Assumptions~\ref{asm:timing} and
  \ref{asm:no_awake} hold.
  Then, ASYMM is equivalent to an instance of Algorithm~\ref{alg:inexactMM} in
  which the selection of nodes $i_h$ satisfies an essentially cyclic rule.
Moreover, if in Algorithm~\ref{alg:async}, $\epsilon_i>0$, $\forall i\in\VV$,
  the total number of primal descent steps $h^k$ is finite. 
  \end{theorem}

\begin{proof}
Define $H_i^k{=}\{t\mid \tau_i^{k-1}{<}t{<}\tau_i^{k},\, i \text{ runs task T1 at } t\}$
for $k=0,1,...$, where $\tau_i^{-1}$ is the first time instant in which node $i$
is awake (doing task T1). Then, define ${H}^k=\bigcup_{i\in\VV}H_i^k$,
$h^k=|{H}^k|$ and let
\begin{equation*}
 \upsilon_1<\upsilon_2<...<\upsilon_{h^k}\
\end{equation*}
be the ordered sequence of elements (time instants) in ${H}^k$.

By setting $i_h = i_{v_h}$, for $h=1,\dots, h^k$, in
Algorithm~\ref{alg:inexactMM} and using Lemma~\ref{lemma:l_consistency} and Lemma~\ref{lemma:x_consistncy}, one has that ASYMM turns
out to be equivalent to an instance of Algorithm~\ref{alg:inexactMM}.
Moreover, from Lemma~\ref{lemma:awake_once} every node runs task T1 at least once in $\{
v_1, \dots, v_{h^k} \}$. Hence, Algorithm~\ref{alg:inexactMM} is run with an essentially
cyclic update rule over a time window of length $h_k$, which is finite
due to Lemma~\ref{lemma:l_cyc}.
\end{proof}

\begin{remark}
  The duration of each ASYMM cycle can be bounded both from below and from above as follows. From the definition of $h^k$ it can be easily verified that $h^k\geq N$. Moreover, from
  Lemma~\ref{lemma:awake_once} one has that $h^k\leq t_{(k+1)N+1}-t_{kN+1}-N$. Hence,
  $t_{(k+1)N+1}-t_{kN+1}\geq 2N$. On the other side, from Proposition~\ref{prop:logic-AND} it follows that $t_{(k+1)N}-t_{kN+1}\leq d_G \bar{T}$ for all $k$.
\end{remark}

\subsection{A bound on the Lagrangian gradient}
In virtue of the equivalence result in Theorem~\ref{thm:equivalence}, ASYMM inherits the convergence properties of Algorithm~\ref{alg:inexactMM} applied to Problem~\ref{pb:distributed_problem}. In particular, in order to guarantee the convergence of the inexact MM to a local minimum, it is necessary that the Lagrangian minimization is asymptotically exact (see, e.g.,~\cite[Section~2.5]{bertsekas2014constrained}).
To this aim, in this section an upper bound $\varepsilon$ on the norm of the gradient of the augmented Lagrangian~\eqref{eq:L} is derived as a function of the local tolerances $\epsilon_i$ employed in task T1 of ASYMM. The result requires (at a given iteration $k$) a technical assumption of local (strong) convexity of the augmented Lagrangian.

Let us introduce the following preliminary result.
\begin{lemma}\label{lemma:local_gradient_bound}
  Let $\Phi(y_1, \ldots, y_N)$ be a $\sigma$-strongly convex function with block
  component-wise Lipschitz continuous gradients (with $L_i$ being the Lipschitz
  constant with respect to block $y_i$) in a subset $Y\subseteq\R^n$.
Let $\{\y^h\}$ be a sequence generated according to $\y^{h+1}=\y^h-\frac{1}{L_{i_h}} U_{i_h}\nabla_{y_{i_h}}\Phi(\y^h)$,
  where $\y^0\in Y$ and indexes $i_h\in\until{N}$ are drawn in an essentially cyclic way.
  If, for some $\bar{h}_i>0$, $\|\nabla_{y_i}\Phi(\y^{\bar{h}_i})\|\leq \epsilon_i, \; \forall i \in
\until{N}$,
  then

  \begin{equation*}
    \|\nabla_{\y}\Phi(\y^h)\|\leq
\sqrt{\sum_{i=1}^N\left(\frac{L_i\epsilon_i}{\sigma}\right)^2}
  \end{equation*}
  for all $h\geq\bar{h}=\max_{i\in\until{N}}\bar{h}_i$.
\end{lemma}
\begin{proof}
  See the Appendix.
\end{proof}
The next result provides a bound on the norm of the gradient of the augmented Lagrangian $\Lag(\x,\LL)$ used in ASYMM.
\begin{proposition}\label{prop:conv}
  Let Assumption~\ref{asm:C2_functions} hold and assume that $\x^k$ generated
  by ASYMM belongs to a neighborhood of a local minimum of $\Lag_{\p^k}(\x,\LL^{k})$ where $\Lag_{\p^k}(\x,\LL^{k})$ is 
  $\sigma^k$-strongly convex. 
  Denote by $\epsilon_i^k$ the local tolerance set by node $i$ for the primal descent during cycle $k$ and by $L_i^k$ the Lipschitz constant of $\nabla_{x_i}\tLag_{\p_{\NN_i}^k}(x_{\NN_i},\LL_{\NN_i}^k)$. Then, it holds $$\|\nabla_{\x}\Lag_{\p^k}(\x^{k+1},\LL^k)\| \leq
   \varepsilon^k
   =\sqrt{\sum_{i=1}^N\left(\frac{L_i^k\epsilon_i^k}{\sigma^k}\right)^2}.$$
\end{proposition}
\begin{proof}
  From Proposition~\ref{prop:Lipschitz}, $\Lag_{\p^k}(\x,\LL^k)$ has block component-wise Lipschitz continuous gradient with constants $L_i^k$.
  Moreover, during the $k$-th cycle of ASYMM, there occurred that 
 $\|\nabla_{x_i}\tLag_{\p_{\NN_i}^k}(x_{\NN_i},\LL_{\NN_i}^k)\|\leq\epsilon_i^k$ and hence, by \eqref{eq:gradient_equivalence}, also $\|\nabla_{x_i} \Lag_{\p^k}(\x,\LL^{k})\|\leq\epsilon_i^k$. Then, being $\x^{k+1}$ generated through a block coordinate descent (according to Theorem~\ref{thm:equivalence}), the proof follows from Lemma~\ref{lemma:local_gradient_bound}.
\end{proof}

Proposition~\ref{prop:conv} relates the local tolerances adopted by the nodes in the primal descent to a global bound on the norm of the gradient of the augmented Lagrangian.
The result requires to assume that the vector $\x^k$ generated during the $k$-th cycle of ASYMM lies in a neighborhood of a local minimum of the current augmented Lagrangian, in which the augmented Lagrangian itself is strongly convex. 
This assumption is indeed strong, but it is somehow standard in the nonconvex optimization literature (see, e.g.,~\cite[Section~2.2.4]{bertsekas2014constrained}). In practice it turns out to be typically satisfied after a sufficient number of iterations of multiplier/penalty parameter updates.
In fact, as $\p^k$ grows according to the update rules \eqref{eq:penalty_eq_x}-\eqref{eq:penalty_ineq}, the augmented Lagrangian tends to become locally strongly convex.
Moreover, from a practical point of view, it has been observed that choosing the obtained $\x^k$ as the initial condition for the $(k{+}1)$-th minimization usually generates sequences $\{\x^k\}$ that remain within a neighborhood of the same local minimum of Problem~\eqref{pb:distributed_problem}, thus meaning that there exists a cycle $\bar{k}$ after which the assumption in Proposition~\ref{prop:conv} will hold indefinitely. Hence, if the assumptions of Proposition~\ref{prop:conv} hold from a certain cycle $\bar{k}$ and the local tolerances $\epsilon_i^k$ vanish as $k\to\infty$, the minimization of the augmented Lagrangian turns out to be asymptotically exact. Therefore, convergence results such those in~\cite[Section~2.5]{bertsekas2014constrained} can be recovered, when applying ASYMM to Problem~\eqref{pb:distributed_problem}. The interested reader is referred to~\cite[Chapter~2]{bertsekas2014constrained} for a thorough discussion on the convergence properties of the Method of Multipliers.

\section{Dealing with big-data optimization}
\label{sec:extensions}
The ASYMM algorithm can be easily amended to deal with so called
  \emph{distributed big-data optimization}, \cite{notarnicola2017distributed},
  i.e., the distributed solution of problems in which $x\in\R^n$ with $n$ very
  large.
  In this set up arising, e.g., in estimation and learning problems, two main problems may arise. On one hand, the primal and multiplier
  update steps may not be executable in a single step by some node because
the computation of the whole gradient in the primal step may be cumbersome. 
Moreover, communication bottlenecks may arise, in fact it may happen for some node $i$,
that the local optimization variable $x_i$ does not fit the communication channels
between node $i$ and its neighbors.

Assume each agent $i$ to partition its local decision variable $x_i$ in $N_i$ blocks, i.e.,
$x_i=\left[x_{i,1} \dots x_{i,N_i}\right]^\top$, where $x_{i,m}=U_{i,m}^\top x_i$
and $x_i=\sum_{m=1}^{N_i}U_{i,m}x_{i,m}$.

Whenever node $i$ wakes up to perform task T1, it computes the primal descent step on one of the $N_i$ blocks (say $m$) instead of performing it on the whole $x_i$, i.e. it computes
\[
x_\im= x_\im-\frac{1}{L_\im}\nabla_{x_\im}\tLag_{\p_{\NN_i}}(x_{\NN_i},\LL_{\NN_i})
\]
where $m$ is picked in an essentially cyclic way.
Similarly the update of the local multiplier vectors can be carried out on one block at a time, e.g.,
\[
\nu_{ij,m}\gets\nu_{ij,m}+\rho_{ij}(x_\im-x_{j,m}).
\]

By following the same reasoning adopted in Sections $3$ and $4$, it can be shown that the primal descent steps are equivalent to a block coordinate descent algorithm on the augmented Lagrangian and converge to a stationary point.

The only additional assumption needed for the convergence result is that
functions $f_i$, $\h_i$, $\g_i$ have block component-wise Lipschitz continuous
gradients.  If for some node $i$, $x_i$ does not fit the communication channels
(the same holds true for $\nu_{ij}$), ASYMM is easily extended by allowing node
$i$ to transmit $x_\im$ and $U_\im$ at the end of each task T1 and to split the
transmission on $\nu_{ij}$ in multiple steps.

\section{Numerical Results}
Two examples are presented to assess the performance of ASYMM. The first one involves nonconvex local constraints, while the second requires the minimization of a nonconvex objective function.
\label{sec:numerical}
\subsection{Distributed source localization}
Consider a network of $N$ sensors, deployed over a certain region, communicating
according to a connected graph $\GG=(\VV,\EE)$, which have to solve the optimization problem
\begin{equation*}
\begin{aligned}
& \m_x
& & \sum_{i=1}^N f_i(x)\\
& \st
& & \|x-c_i\|-R_i\leq 0, & i=1,...,N\\
& & & r_i-\|x-c_i\|\leq 0, & i=1,...,N,
\end{aligned}
\end{equation*}
which can be rewritten in the form of problem~\eqref{pb:distributed_problem}.

\begin{figure}
\centering
\begin{subfigure}{0.49\textwidth}
		\centering
		\includegraphics[width=0.8\linewidth]{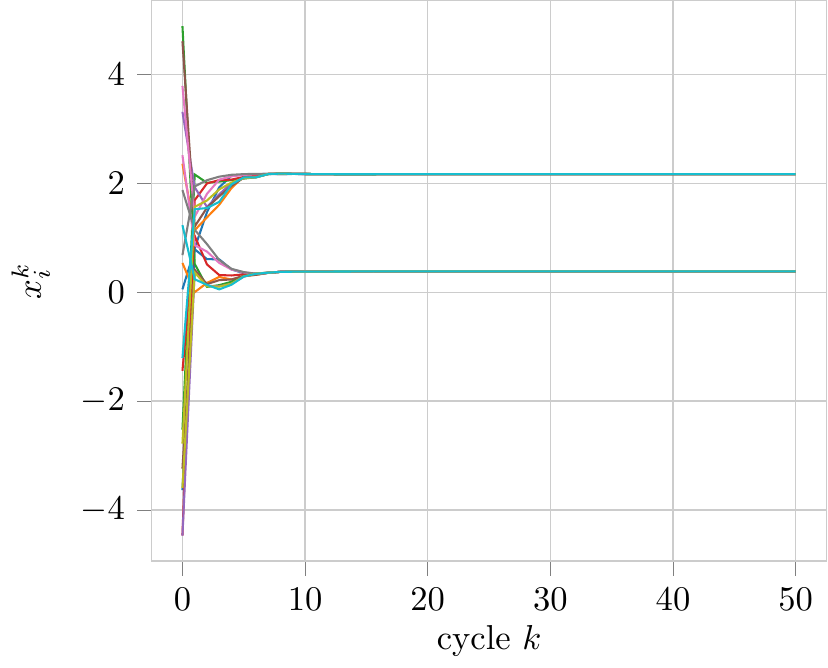}
	\caption{Evolution of the local decision variables $x_i^k$.}
        \label{fig:xk}
\end{subfigure}
\begin{subfigure}{0.49\textwidth}
	\centering
	\includegraphics[width=0.8\linewidth]{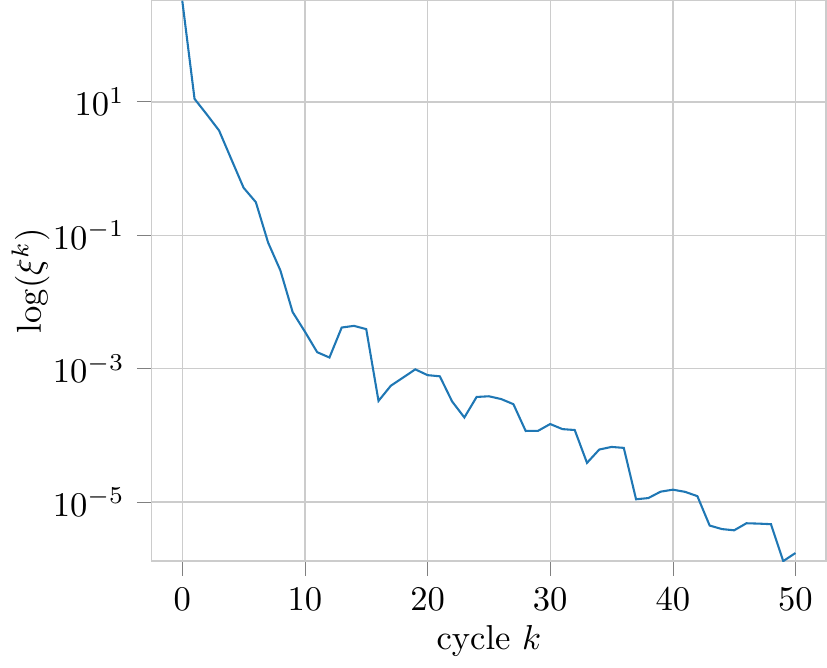}
	\caption{Logarithm of the measure of infeasibility $\xi^k$.}
	\label{fig:infeasibility}
	\end{subfigure}
	\caption{Distributed source localization}
\end{figure}

Such a problem naturally arises, for example, in the context of source
localization, in which each agent knows its own absolute location $c_i$ and
takes a noisy measurement $y_i$ of its own distance from an an emitting source
located at an unknown location $x^\star$ (for example through a laser) as
$ y_i=\|x^\star-c_i\|+w_i$.
If we make the assumption of unknown but bounded (UBB) noise, i.e. for all
$i=1,...,N$, the noise signal $w_i$ satisfies $|w_i|\leq \kappa_i$
for some $\kappa_i\geq 0$, then each node is able to define its own feasible set
$X_i$ in which the unknown source location must lie as $X_i=\{x\mid r_i \leq\|x-c_i\|\leq R_i\}$,
where $r_i=y_i-\kappa_i$ and $R_i=y_i+\kappa_i$. Notice that each set $X_i$ is a
circular crown and hence it is a non convex set.

Suppose $f_i(x_i)=x_i^\top x_i$ for all $i\in\VV$. We report a simulation with
$N=10$ nodes and $n=2$, in which $x^\star\in U[-2.5,2.5]^n$,
$c_i\in U[-2.5,2.5]^n$ and $\kappa_i\in U[0,0.3]$ for all $i\in\VV$, where $U[a,b]$
denotes the uniform distribution between $a$ and $b$. The graph is modeled
through a connected Watts-Strogatz model in which nodes has mean degree $K=2$.
Let us define the measure of infeasibility at iteration $k$ as
 \begin{align*}
 	\xi^k=&\sum_{i=1}^N\bigg(\max(0,\|x_i^k-c_i\|-R_i)+\\
 	&+\max(0,r_i-\|x_i^k-c_i\|)+\sum_{j\in \NN_i}\|x_i^k-x_j^k\|\bigg)
 \end{align*}

We run ASYMM for $25000$ iterations with $\beta=4$ and $\gamma=0.25$.  
Figure~\ref{fig:xk} shows the evolution
of $x_i^k$ for each $i\in\VV$. Finally, in Figure~\ref{fig:infeasibility} the
values of $\xi^k$ are reported. As it can be seen, the nodes performed $50$
multiplier updates each, along the $25000$ iterations (corresponding to 2500 awakenings per node on average).

\subsection{Distributed nonlinear classification}
In this example we consider a nonlinear classification problem in which the data to be classified are represented as points $z\in\R^2$ which belong to two different classes. So, each point is associated a label $y\in\{-1,1\}$, which represents the class the point belongs to.

The considered classifier can be represented as a Neural Network (NN) consisting of one input layer with two units, two hidden layers with four and two units respectively, and an output layer with one unit (respectively green, blue and red in Figure~\ref{fig:NN}). Moreover, a bias unit is present in both the input and the hidden layers (in yellow in Figure~\ref{fig:NN}). Define $w_1\in\R^{2\times 4}$, $b_1\in\R^{4}$, $w_2\in\R^{4\times 2}$, $b_2\in\R^2$, $w_3\in\R^{2\times1}$ and $b_3\in\R$. Moreover, define $x\in\R^{25}$ as the stack of all the previously defined variables.

Define the output of the first hidden layer as
\begin{equation}
	l_1(z,w_1,b_1)=\tanh(w_1^\top z + b_1).
\end{equation}
where the operator $\tanh$ is to be intended component-wise.
Similarly, the output of the second hidden layer is
\begin{equation}
	l_2(z,w_1,b_1,w_2,b_2)= \tanh(w_2^\top l_1(z,w_1,b_1) + b_2)
\end{equation}
and the output of the whole NN is
\begin{equation}\label{eq:class}
	f(z,x)= \tanh(w_3^\top l_2(z,w_1,b_1,w_2,b_2) + b_3)
\end{equation}
\begin{figure}
	\centering
	\includegraphics[width=0.5\linewidth]{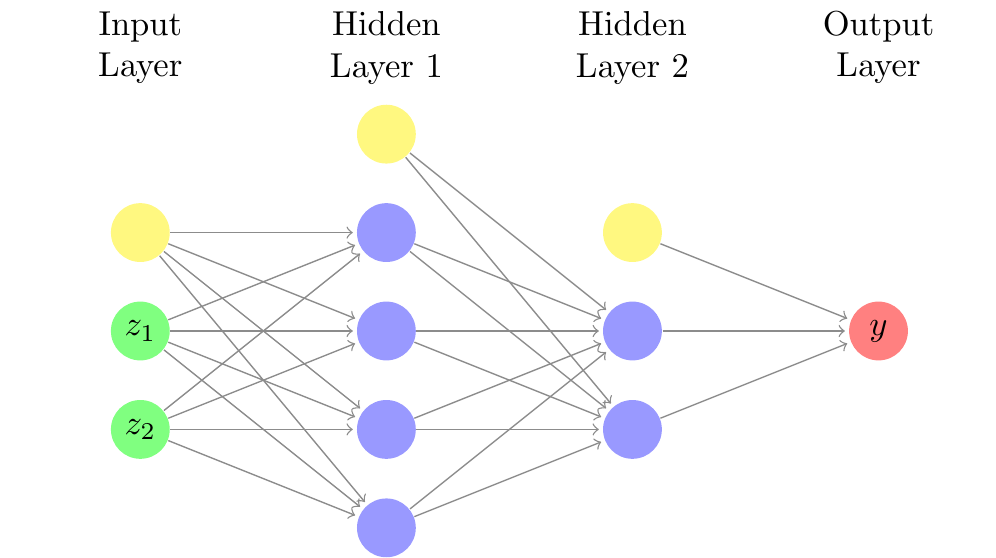}
	\caption{Graphical representation of the Neural Network used as classifier. In green the input units, in blue the hidden units, in red the output unit and in yellow the bias units.}
	\label{fig:NN}
\end{figure}
\begin{figure}
\begin{subfigure}{0.49\textwidth}
	\centering
	\includegraphics[width=0.9\linewidth]{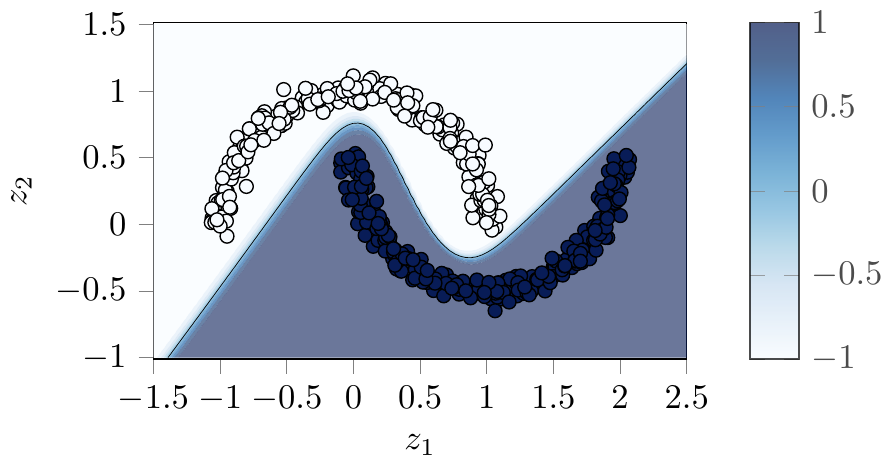}
	\caption{Two moons dataset.}
	\label{fig:classification}
\end{subfigure}
\begin{subfigure}{0.49\textwidth}
	\centering
	\includegraphics[width=0.9\linewidth]{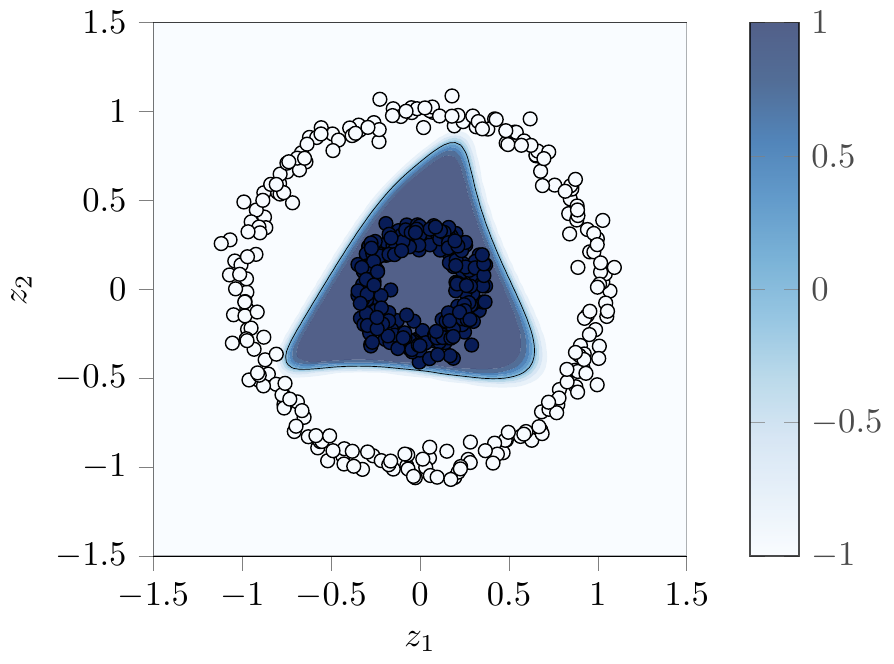}
	\caption{Nested circles dataset.}
	\label{fig:classification_circ}
	\end{subfigure}
	\caption{Distributed nonlinear classification. Blue dots represent points with label $-1$ and white dots points with label $1$. Colored regions represent the output of the classifier~\eqref{eq:class} resulting from the solution of~\eqref{pb:ML_d} provided by ASYMM. The color of the regions is associated to a number as in the color bar.}
\end{figure}

Given a set of labeled data points $(Z,Y)$, the classification problem can be written as
\begin{equation}\label{pb:ML}
\m_x \sum_{z,y\in (Z,Y)} \left( f(z,x) - y \right)^2.
\end{equation}

Suppose now that the dataset is distributed among $N$ nodes, which communicate
according to a connected graph $\GG=(\VV,\EE)$. Each node $i$ owns a portion of the dataset $(Z_i, Y_i)$, which must remain private and cannot be shared with the other nodes. In this framework, problem~\eqref{pb:ML} can be rewritten in the equivalent form
\begin{equation}\label{pb:ML_d}
\begin{aligned}
& \m_{x_1,\dots,x_N} 
& & \sum_{i=1}^N \sum_{z,y\in (Z_i,Y_i)} \left( f(z,x_i) - y \right)^2\\
& \st
& & x_i=x_j, \qquad\forall (i,j)\in\EE
\end{aligned}
\end{equation}

In our simulations two datasets are considered, which are benchmarks used in the context of machine learning. The first one consists of points belonging to two moon-shaped subsets (see Figure~\ref{fig:classification}), in which points of one subset have label $y=1$, points of the other have label $y=-1$. In the second one, data points are distributed along two nested circles. Points on the inner circle have label $y=1$, while the others have label $y=-1$ (see Figure~\ref{fig:classification_circ}).

The ASYMM algorithm has been run on $N=10$ nodes, each one processing a local dataset $(Z_i, Y_i)$ consisting of 100 points. The obtained classifiers are represented in Figures~\ref{fig:classification} and~\ref{fig:classification_circ}. The colored regions represent the value of~\eqref{eq:class} computed in those points at the (local) minimum $x^\star$ obtained by ASYMM.

It is worth stressing that the example presents a low-size classification problem, with the purpose of illustrating the proposed technique. When massive data in higher dimensional spaces are available, it is necessary to consider a more complex neural network (i.e., with a much higher number of neurons) so that the dimension of the decision variable can be fairly high. In such a case, the big-data approach proposed in Section~\ref{sec:extensions} can be adopted.

\section{Conclusions}
\label{sec:conclusions}
In this paper, an asynchronous distributed algorithm for nonconvex 
optimization problems over networks has been proposed. By suitably 
defining local augmented Lagrangian functions, the optimization process 
has been distributed among the agents of the network. A fully asynchronous 
implementation has been devised, taking advantage of a distributed 
logic-AND algorithm that allows the agents to regulate the sequence of primal 
and dual update steps. The proposed ASYMM algorithm is shown to be 
equivalent to an inexact version of the centralized method of 
multipliers, thus inheriting its main properties.
An extension to big-data problems, 
featuring high-dimensional decision variables, has been also presented. 

Ongoing research concerns the specialization of the proposed method to 
different application domains, including distributed set membership 
estimation and machine learning with constraints.

\bibliographystyle{IEEEtran}
\bibliography{biblio}

\section*{Appendix}
\section*{Proof of Lemma~\ref{lemma:local_gradient_bound}}
Consider a function $\Phi(y)$, $\sigma$-strongly convex in a subset $Y\subseteq\R^n$, with $L$-Lipschitz continuous gradient. Define $y^\star=\arg\min_{y\in Y}\Phi(y)$.
From the definition, if $\Phi(y)$ is $\sigma$-strongly convex, then, using the Cauchy Schwartz inequality one obtains
\begin{equation*}
\|\nabla\Phi(y) - \nabla\Phi(z)\|\geq \sigma\|y-z\|.
\end{equation*}
Then, it can be easily proved that
\begin{equation}\label{eq:strong_Lip}
	\sigma \|y-y^\star\|\leq \|\nabla\Phi(y)\|\leq L\|y-y^\star\|,\;\forall y\in Y
\end{equation}

In order to prove Lemma~\ref{lemma:local_gradient_bound} the following technical results are needed. 
\begin{lemma}\label{lemma:y_decreasing}
	Performing a gradient descent algorithm on $\Phi(y)$, starting from $y^0$ and using a step-size equal to $\frac{1}{L}$, i.e.
	\begin{equation}\label{eq:Lip_update}
		y^{h+1}=y^h-\frac{1}{L}\nabla_y\Phi(y^h)
	\end{equation}
	produces a sequence $\{y^h\}$ such that,
	\begin{equation}\label{eq:y_dec}
		\|y^{h+1}-y^\star\|\leq\|y^h-y^\star\|
	\end{equation}
\end{lemma}
\begin{proof}
  See, e.g.,~\cite[Theorem 2.1.14]{nesterov2013introductory}.
\end{proof}
\begin{lemma}\label{thm:gradient_bound}
	Consider a sequence $\{y^h\}$ generated as $y^{h+1}=y^h-\frac{1}{L}\nabla \Phi(y^h)$ with $y^0\in Y$.
	Then:
	\begin{enumerate}
		\item \label{pt:1} for all $h\geq \bar{h}$ it holds that
		\begin{equation}\label{eq:decreasing_gradient}
		\|\nabla\Phi(y^h)\| \leq L\|y^{\bar{h}}-y^\star\|
		\end{equation}
		\item if for some $\bar{h}$ it holds that $\|\nabla \Phi(y^{\bar{h}})\|=\varepsilon$, then
		\begin{equation}\label{eq:nphi}
		\|\nabla \Phi(y^h)\|\leq \frac{L\varepsilon}{\sigma}
		\end{equation}
		for all $h\geq \bar{h}$.
	\end{enumerate}
\end{lemma}
\begin{proof}
	Using the right side of~\eqref{eq:strong_Lip} 
		and Lemma~\ref{lemma:y_decreasing} one has that
		\begin{equation*}
		\|\nabla\Phi(y^{\bar{h}+1})\|\leq L\|y^{\bar{h}+1}-y^\star\|\leq L\|y^{\bar{h}}-y^\star\|
		\end{equation*}
		By induction,~\eqref{eq:decreasing_gradient} follows directly and this concludes the proof for point~\ref{pt:1}.
		For point (ii), since $\|\nabla \Phi(y^{\bar{h}})\|=\varepsilon$, from the left side of~\eqref{eq:strong_Lip}, one has that
		\begin{equation*}
		\sigma \|y^{\bar{h}}-y^\star\|\leq\varepsilon
		\end{equation*}
		which can be rewritten as
		\begin{equation}\label{eq:low_bd}
		\|y^{\bar{h}}-y^\star\|\leq\frac{\varepsilon}{\sigma}
		\end{equation}
		Then, substituting~\eqref{eq:low_bd} in the right side of~\eqref{eq:strong_Lip}, we obtain~\eqref{eq:nphi}
		which, from point~\ref{pt:1} concludes the proof.
\end{proof}
\begin{figure}
	\centering
	\includegraphics[width=0.5\linewidth]{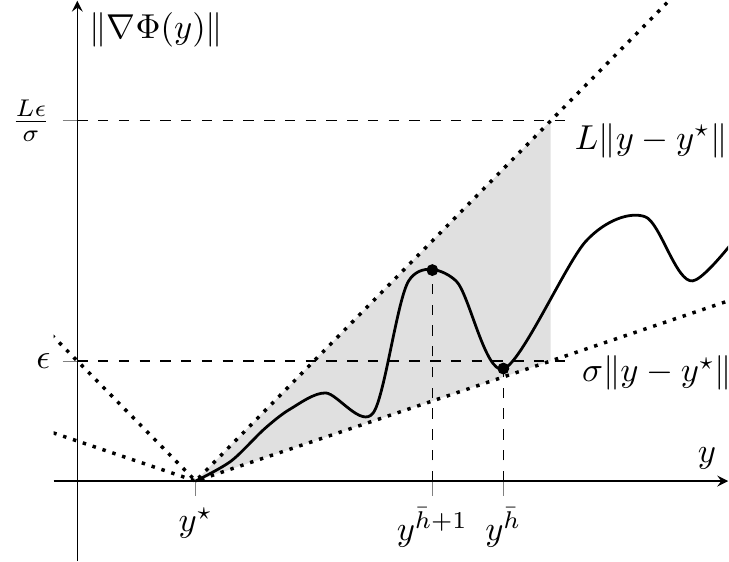}
	\caption{Representation of the results of Lemma~\ref{thm:gradient_bound}.}
	\label{fig:lipschitz}
\end{figure}
A graphical representation of the previous
Lemma is given in Figure~\ref{fig:lipschitz}. The gradient of $\Phi(y)$ is bounded by the dotted lines, as from~\eqref{eq:strong_Lip}. Moreover, from~\eqref{eq:y_dec}, given $\norm{\nabla_y\Phi(y^{\bar{h}})}$, it holds that $\norm{\nabla_y\Phi(y^{\bar{h+1}})}$ stays in the shaded region.

Finally, Lemma~\ref{lemma:local_gradient_bound} is proved by noting that, from Lemma~\ref{thm:gradient_bound} 
\begin{equation*}
\|\nabla_{y_i}\Phi(\y^{{h}})\|\leq \frac{L_i\epsilon_i}{\sigma}
\end{equation*}
for all $h\geq\bar{h}_i$, and
\begin{equation*}
\|\nabla_{\y}\Phi(\y^h)\|= \sqrt{\sum_{i=1}^N \|\nabla_{y_i}\Phi(\y^{{h}})\|^2}.
\end{equation*}

\end{document}